\documentclass[11pt]{amsart}

\usepackage{amsmath,amssymb,amsthm,mathtools}
\usepackage{hyperref}
\usepackage{enumitem}

\newcommand{\wop}{w.o.p.\ }
\newcommand{\len}[1]{\lvert #1\rvert}
\newcommand{\eps}{\varepsilon}

\DeclareMathOperator{\Area}{Area}

\theoremstyle{plain}
\newtheorem{theorem}{Theorem}[section]
\newtheorem{lemma}[theorem]{Lemma}
\newtheorem{corollary}[theorem]{Corollary}
\newtheorem{proposition}[theorem]{Proposition}

\theoremstyle{definition}
\newtheorem{definition}[theorem]{Definition}
\newtheorem{remark}[theorem]{Remark}

\title[RANDOM GROUPS AT $d<1/2$: SHARP LENGTH INEQUALITIES]%
{RANDOM GROUPS AT DENSITY $d<1/2$: SHARP LENGTH INEQUALITIES FOR\\
GENERALIZED TORSION AND A FIXED-WIDTH EXCLUSION VIA FIRST-ORDER TRANSFER}

\author{HYUNGRYUL BAIK}

\address{Department of Mathematical Sciences, KAIST, 291 Daehak-ro, Yuseong-gu, Daejeon, 34141, South Korea}
\email{hrbaik@kaist.ac.kr}


\begin{document}

\begin{abstract}
Let $G$ be a random group in Gromov's density model $G(m,d,L)$ with $d<\tfrac12$.
We prove a sharp quantitative constraint on products of conjugates equal to the identity:
for every $n\ge1$ and $\varepsilon>0$, with overwhelming probability as $L\to\infty$,
any tight word
\[
W=\prod_{i=1}^n h_i^{-1} g h_i =1 \quad\text{in } G
\]
(with $g\neq 1$ as a word) satisfies the inequality
\[
\sum_{i=1}^n \len{h_i} \;>\; \frac{1-2d-\varepsilon}{2}\,L \;-\; \frac{n}{2}\,\len{g}.
\]
The proof is a short van Kampen diagram argument: Ollivier's sharp isoperimetric inequality
forces a 2-cell contributing a large portion of its boundary to the outer boundary, and a
simple boundary block-counting estimate yields this corridor-type lower bound.
As consequences we obtain uniform short-witness exclusions and width--length tradeoffs for generalized torsion
at every density $d<\tfrac12$.
We also deduce that random groups have no generalized torsion of any fixed width as a corollary of the
recent first-order transfer theorem of Kharlampovich, Miasnikov, and Sklinos.
\end{abstract}

\maketitle

\section{Introduction}

Gromov's density model provides a framework for studying ``typical'' finitely presented groups.
For density $d<\tfrac12$, random groups are non-elementary hyperbolic and satisfy linear isoperimetric inequalities.
A particularly strong form is due to Ollivier: for every $\eta>0$, with overwhelming probability as $L\to\infty$,
every reduced van Kampen diagram $\Delta$ satisfies
\[
\len{\partial\Delta}\ \ge\ (1-2d-\eta)\,L\cdot \Area(\Delta)
\]
\cite{Ollivier2004,Ollivier2007}.
(Throughout, \wop means probability tending to $1$ as $L\to\infty$ for fixed $m,d$.)

We are interested in \emph{generalized torsion}.
A nontrivial element $g\in G$ has generalized torsion of width $n$ if
\[
(h_1^{-1}gh_1)\cdots(h_n^{-1}gh_n)=1
\]
for some $h_1,\dots,h_n\in G$.
This is an obstruction to bi-orderability (free groups are generalized torsion-free), but it is not ruled out
by hyperbolicity: torsion-free hyperbolic groups may contain generalized torsion
\cite{ItoMotegiTeragaito2023}. While it is known that random groups are not left-orderable \cite{Orlef2017}, 
quantitative constraints on generalized torsion remain of interest.

\medskip\noindent
\textbf{Main quantitative result.}
Our main contribution is a sharp length inequality for products of conjugates equal to the identity at every density $d<\tfrac12$.
The proof combines Ollivier's sharp isoperimetry with a short boundary counting argument:
sharp isoperimetry forces a 2-cell contributing many boundary edges, and a worst-case cap on the number of
these edges that can lie in the $g$-blocks yields the inequality.

\begin{theorem}[Sharp length inequality at density $d<1/2$]\label{thm:main}
Fix $d<1/2$, $n\ge 1$, and $\varepsilon>0$.
In the density model $G(m,d,L)$, \wop as $L\to\infty$, the following holds.

If a tight word
\[
W=\prod_{i=1}^n h_i^{-1} g h_i
\]
(with $g\neq 1$ as a word, and no cancellation across parentheses after cyclic rotation)
represents $1$ in the random group, then
\begin{equation}\label{eq:main-ineq}
\sum_{i=1}^n \len{h_i} \;>\; \frac{1-2d-\varepsilon}{2}\,L \;-\; \frac{n}{2}\,\len{g}.
\end{equation}
\end{theorem}

In Section~\ref{sec:random} we derive two quick consequences: a uniform short-witness exclusion and a width--length tradeoff
(Corollaries~\ref{cor:short}--\ref{cor:tradeoff}).

\medskip\noindent
\textbf{A general diagrammatic principle and a logical corollary.}
Section~\ref{sec:general} isolates the underlying mechanism: the same inequality holds in any presentation with fixed relator length
and a strict linear isoperimetric inequality (Theorem~\ref{thm:liniso-ineq}).
Finally, using the first-order transfer theorem of Kharlampovich--Miasnikov--Sklinos \cite{KMS2025}
(see also \cite{KharlampovichSklinos2021,KharlampovichSklinos2022}), we deduce a fixed-width exclusion
for random groups (Corollary~\ref{cor:KMS}).


\section{A general estimate from a large boundary face}\label{sec:general}

\subsection*{Tight conjugate normal form}
\begin{definition}[Tight conjugate normal form]\label{def:tight}
A boundary label $W$ is in \emph{tight conjugate normal form of width $n$} if it is written as
\[
W=\prod_{i=1}^n h_i^{-1} g h_i
\]
where each of $g,h_1,\dots,h_n$ is freely reduced, $g\neq 1$ as a word, and there is no cancellation across the parentheses
(cyclically, after rotating the product).
\end{definition}

\begin{remark}
Any equality of the form $\prod_{i=1}^n(h_i^{-1}gh_i)=1$ can be tightened into the form of
Definition~\ref{def:tight} by free reduction and cyclic rotation, without increasing $\len{g}$ or $\sum_i\len{h_i}$.
We therefore work in tight form throughout.
Crucially, if $W$ is tight, it is **cyclically reduced** as a word. This ensures that any minimal van Kampen diagram for $W$ is reduced (i.e., contains no spurs/filaments), allowing us to apply isoperimetric inequalities for reduced diagrams.
\end{remark}

\subsection*{A lower bound from one large boundary face}
\begin{lemma}[Lower bound from one large boundary face]\label{lem:oneface}
Let $W=\prod_{i=1}^n h_i^{-1} g h_i$ be in tight conjugate normal form and suppose $W=1$ in a group presentation.
Let $\Delta$ be a reduced van Kampen diagram for $W$.
If there exists a 2-cell $D$ with
\[
\len{\partial D\cap \partial\Delta}>\alpha,
\]
then
\begin{equation}\label{eq:oneface-ineq}
\sum_{i=1}^n \len{h_i} \;>\; \frac{\alpha}{2}-\frac{n}{2}\len{g}.
\end{equation}
\end{lemma}

\begin{proof}
Let $q:=\partial D\cap\partial\Delta$.
Along $\partial\Delta$, the boundary word $W$ consists of $n$ many $g^{\pm1}$-blocks (each of length $\len{g}$)
and $2n$ many $h_i^{\pm1}$-blocks (total length $2\sum_i\len{h_i}$).

Let $S_g(q)$ be the number of edges of $q$ lying in the $g^{\pm1}$-blocks, and $S_h(q)$ those lying in the $h_i^{\pm1}$-blocks.
Then $S_g(q)+S_h(q)=\len{q}$.
Since $q$ is a subpath of $\partial\Delta$, the number of $g$-edges in $q$ cannot exceed the total number of $g$-edges in $\partial\Delta$.
Hence $S_g(q)\le n\len{g}$.
It follows that
\[
S_h(q)=\len{q}-S_g(q)>\alpha-n\len{g}.
\]
But $S_h(q)\le 2\sum_i \len{h_i}$, as $S_h(q)$ is a subset of the total $h$-length on $\partial\Delta$.
Thus $2\sum_i\len{h_i}>\alpha-n\len{g}$, which is~\eqref{eq:oneface-ineq}.
\end{proof}

\begin{remark} 
The boundary portion $q$ contributed by one 2-cell controls the total $h$-length needed to ``support'' the word,
using only the worst-case cap $S_g(q)\le n\len{g}$.
We do not require $q$ to be a single arc; only its total length matters.
\end{remark}

\subsection*{A strict linear-isoperimetry consequence (small generalization)}
\begin{definition}
\label{def:liniso}
Let $\langle X\mid R\rangle$ be a presentation in which every relator has length exactly $L$.
Fix $\beta\in(0,1]$.
We say the presentation satisfies a \emph{strict $(\beta,L)$-linear isoperimetric inequality} if every reduced
van Kampen diagram $\Delta$ with $\Area(\Delta)\ge 1$ satisfies
\[
\len{\partial\Delta}\ >\ \beta L\cdot \Area(\Delta).
\]
\end{definition}

\begin{lemma}[A large boundary face from strict linear isoperimetry]\label{lem:largeface}
Assume every relator has length exactly $L$.
If a reduced van Kampen diagram $\Delta$ satisfies $\len{\partial\Delta}>\beta L\cdot \Area(\Delta)$ for some $\beta>0$,
then there exists a 2-cell $D$ such that
\[
\len{\partial D\cap \partial\Delta}>\beta L.
\]
\end{lemma}

\begin{proof}
Each boundary edge of $\partial\Delta$ is incident to a unique 2-cell.
Hence
\[
\sum_{D}\len{\partial D\cap \partial\Delta}=\len{\partial\Delta}.
\]
Dividing by $\Area(\Delta)$, the average boundary contribution per face equals
$\len{\partial\Delta}/\Area(\Delta)$.
If this average is $>\beta L$, then some face satisfies $\len{\partial D\cap\partial\Delta}>\beta L$.
\end{proof}

\begin{theorem}[Inequality from strict linear isoperimetry]\label{thm:liniso-ineq}
Assume $\langle X\mid R\rangle$ has relators of length exactly $L$ and satisfies a strict $(\beta,L)$-linear isoperimetric inequality
in the sense of Definition~\ref{def:liniso}.
If $W=\prod_{i=1}^n h_i^{-1} g h_i$ is in tight conjugate normal form and represents $1$ in $G=\langle X\mid R\rangle$, then
\[
\sum_{i=1}^n \len{h_i} \;>\; \frac{\beta}{2}\,L \;-\; \frac{n}{2}\,\len{g}.
\]
\end{theorem}

\begin{proof}
Let $\Delta$ be a reduced van Kampen diagram for $W$.
By assumption, $\len{\partial\Delta}>\beta L\Area(\Delta)$, and $\Area(\Delta)\ge 1$ since $g\neq 1$ and $W$ is tight.
Apply Lemma~\ref{lem:largeface} to obtain a 2-cell $D$ with $\len{\partial D\cap\partial\Delta}>\beta L$,
then apply Lemma~\ref{lem:oneface} with $\alpha=\beta L$.
\end{proof}

\subsection*{A deterministic $C'(1/6)$ comparison}
We include the classical deterministic special case, which yields a direct small-cancellation analogue.

\begin{proposition} 
\label{prop:Cprime}
Assume a finite presentation $\langle X\mid R\rangle$ satisfies $C'(1/6)$ and put $L_{\min}=\min\{\len{r}:r\in R\}$.
If $W=\prod_{i=1}^n h_i^{-1} g h_i$ (tight conjugate normal form) represents $1$ in $G$, then
\[
\sum_{i=1}^n \len{h_i} \;>\; \frac14\,L_{\min} \;-\; \frac{n}{2}\,\len{g}.
\]
In particular, if $\sum_i\len{h_i}<\frac14 L_{\min}$ then necessarily $\len{g}\ge \frac{1}{4n}L_{\min}$.
\end{proposition}

\begin{proof}
Let $\Delta$ be a reduced van Kampen diagram for $W$.
Greendlinger's lemma under $C'(1/6)$ yields a 2-cell $D$ with an outer arc $q\subset\partial\Delta$ of length
$\len{q}>\frac12 L_{\min}$ (e.g.\ \cite[Ch.\ V]{LyndonSchupp}).
Hence $\len{\partial D\cap\partial\Delta}>\frac12 L_{\min}$.
Apply Lemma~\ref{lem:oneface} with $\alpha=\frac12 L_{\min}$.
\end{proof}

\begin{remark} 
In many standard nondegenerate cases (e.g.\ cyclically reduced boundary and $\Area(\Delta)\ge 2$),
one obtains two distinct faces each contributing an outer arc $>\frac12 L_{\min}$.
Combining the two (disjointness of boundary edges) improves the constant to
$\sum_i\len{h_i}>\frac12 L_{\min}-\frac{n}{2}\len{g}$.
\end{remark}

\section{Random groups at density $d<1/2$ (proof of Theorem~\ref{thm:main})}\label{sec:random}

We now specialize to Gromov's density model $G(m,d,L)$ and prove Theorem~\ref{thm:main} and its consequences.

\begin{lemma}[Ollivier \cite{Ollivier2004,Ollivier2007}]\label{lem:ollivier}
Fix $d<1/2$ and $\eta>0$.
In $G(m,d,L)$, \wop as $L\to\infty$, every reduced van Kampen diagram $\Delta$ satisfies
\begin{equation}\label{eq:ollivier}
\len{\partial\Delta} \;\ge\; (1-2d-\eta)\,L \cdot \Area(\Delta).
\end{equation}
\end{lemma}

\begin{proof}[Proof of Theorem~\ref{thm:main}]
Fix $d<1/2$, $n\ge1$, and $\eps>0$.
Let $G$ be a random group in $G(m,d,L)$.
Let $W=\prod_{i=1}^n h_i^{-1} g h_i$ be tight and suppose $W=1$ in $G$.
Let $\Delta$ be a reduced van Kampen diagram for $W$.

Apply Lemma~\ref{lem:ollivier} with $\eta=\eps/2$.
Then \wop
\[
\len{\partial\Delta}\ \ge\ (1-2d-\eps/2)\,L\cdot \Area(\Delta)
\ >\ (1-2d-\eps)\,L\cdot \Area(\Delta),
\]
so the random presentation satisfies the strict $(\beta,L)$-linear isoperimetric inequality of
Definition~\ref{def:liniso} with $\beta=1-2d-\eps$.
Apply Theorem~\ref{thm:liniso-ineq} to obtain~\eqref{eq:main-ineq}.
\end{proof}

\begin{corollary}[Short-witness exclusion]\label{cor:short}
Fix $d<1/2$, $n\ge 1$, and $\eps>0$.
In $G(m,d,L)$, \wop as $L\to\infty$, there is no tight relation
$W=\prod_{i=1}^n h_i^{-1} g h_i=1$ with
\[
\len{g}\le \frac{1-2d-\eps}{2n}\,L
\qquad\text{and}\qquad
\sum_{i=1}^n \len{h_i} \le \frac{1-2d-\eps}{4}\,L.
\]
\end{corollary}

\begin{proof}
Assume $\len{g}\le \frac{1-2d-\eps}{2n}L$.
Then Theorem~\ref{thm:main} gives
\[
\sum_{i=1}^n\len{h_i}>\frac{1-2d-\eps}{2}L-\frac{n}{2}\cdot\frac{1-2d-\eps}{2n}L
=\frac{1-2d-\eps}{4}L,
\]
which contradicts the assumption that $\sum_{i=1}^n \len{h_i} \le \frac{1-2d-\eps}{4}L$.
\end{proof}

\begin{corollary}[Width--length tradeoff]\label{cor:tradeoff}
Fix $d<1/2$ and $\eps>0$.
In $G(m,d,L)$, \wop as $L\to\infty$, any tight relation $W=\prod_{i=1}^n h_i^{-1} g h_i=1$ satisfies
\begin{equation}\label{eq:tradeoff}
n \;>\; \frac{(1-2d-\eps)L - 2\sum_{i=1}^n \len{h_i}}{\len{g}}.
\end{equation}
In particular, if $\len{g}$ is bounded independently of $L$ and $\sum \len{h_i}\le cL$ with
$c<\tfrac12(1-2d-\eps)$, then necessarily $n$ grows linearly in $L$.
\end{corollary}

\begin{proof}
Rearrange~\eqref{eq:main-ineq} as
$n\len{g}>(1-2d-\eps)L-2\sum_i\len{h_i}$ and divide by $\len{g}>0$.
\end{proof}

\begin{remark} 
At $d<1/12$ we have \wop $C'(1/6)$ and $L_{\min}=L$, and Proposition~\ref{prop:Cprime} gives a deterministic comparison.
Theorem~\ref{thm:main} extends the quantitative obstruction to every $d<1/2$, with the sharp constant $1-2d$.
\end{remark}

\section{First-order transfer and a fixed-width corollary}\label{sec:FO}

We work in the language of groups $L_{\mathrm{grp}}=\{\cdot,^{-1},1\}$.
For each $n\ge1$ consider the sentence
\[
\varphi_n := \exists g\neq 1\ \exists h_1,\dots,h_n\ \bigl((h_1^{-1}gh_1)\cdots(h_n^{-1}gh_n)=1\bigr).
\]
Then $G\models \varphi_n$ iff $G$ contains a width-$n$ generalized torsion element.

Nonabelian free groups are bi-orderable (e.g.\ via the Magnus order), hence generalized torsion-free.
Therefore $F_m\models \neg\varphi_n$ for all $n\ge1$.

We use the following transfer theorem.

\begin{theorem}[Kharlampovich--Miasnikov--Sklinos \cite{KMS2025}]\label{thm:KMS}
Fix $d<1/2$.
A first-order sentence $\sigma$ holds in a random group (density $d$) \wop iff it holds in a nonabelian free group.
\end{theorem}

\begin{corollary}[Fixed-width exclusion via first-order transfer]\label{cor:KMS}
Fix $n\ge 1$ and $d<1/2$.
In Gromov's density model $G(m,d,L)$, a random group $G$ has, \wop as $L\to\infty$,
no generalized torsion of width $n$.
\end{corollary}

\begin{proof}
Since $F_m\models \neg\varphi_n$, Theorem~\ref{thm:KMS} implies that $\neg\varphi_n$ holds in random groups \wop at density $d<1/2$.
\end{proof}

\begin{remark}
Corollary~\ref{cor:KMS} is qualitative (it rules out each fixed width), while Theorem~\ref{thm:main} provides an explicit
length inequality that applies to any potential witness and yields quantitative constraints on width under length restrictions.
\end{remark}

\section*{Acknowledgements}
The author thanks Khanh Le and Wonyong Jang for many fruitful conversations.
This work was supported by the National Research Foundation of Korea (NRF) grant funded by the Korean government (MSIT)
(No.\ RS-2025-00513595).


\end{document}